\documentclass[11pt]{article}

\usepackage[T1]{fontenc}
\usepackage[utf8]{inputenc}
\usepackage{lmodern}
\usepackage{microtype}
\usepackage[a4paper,margin=1in]{geometry}
\usepackage{amsmath,amssymb,amsthm,mathtools}
\usepackage{array,booktabs,longtable}
\usepackage{graphicx}
\usepackage{enumitem}
\usepackage{hyperref}
\usepackage{xcolor}
\definecolor{mutedlink}{RGB}{45,58,84}

\hypersetup{
  colorlinks=true,
  linkcolor=mutedlink,
  citecolor=mutedlink,
  urlcolor=mutedlink,
  pdftitle={Numerical topology of the clique complex of the partition graph: Euler characteristic, clique counts, and sequence data},
  pdfauthor={Fedor B. Lyudogovskiy}
}

\setlist[enumerate]{itemsep=0.25em, topsep=0.35em, leftmargin=2.2em}
\setlist[itemize]{itemsep=0.25em, topsep=0.35em, leftmargin=2.0em}

\newtheorem{proposition}{Proposition}[section]
\newtheorem{problem}{Problem}[section]
\newtheorem{remark}{Remark}[section]
\newtheorem{definition}{Definition}[section]

\newcommand{\Par}{\operatorname{Par}}
\newcommand{\Cl}{\operatorname{Cl}}

\title{Numerical topology of the clique complex of the partition graph:\\[0.2em]
\large Euler characteristic, clique counts, and sequence data}
\author{Fedor B. Lyudogovskiy}
\date{}

\begin{document}
\maketitle

\begin{abstract}
We study the numerical topology of the clique complex
\[
K_n=\Cl(G_n),
\]
where $G_n$ is the partition graph on the set of integer partitions of $n$. Building on the previously established homotopy equivalence
\[
K_n \simeq \bigvee^{\,b_n} S^2,
\]
we shift the focus from qualitative topology to its numerical content. Our main objects are the Euler characteristic $\chi(K_n)$, the derived sequence $b_n=\chi(K_n)-1$, the clique counts $c_r(n)$, and several related maximal-simplex counts.

We develop two exact counting languages for the same invariant. The first is the direct clique-counting formula
\[
\chi(K_n)=\sum_{r\ge 1}(-1)^{r-1}c_r(n),
\]
which expresses Euler characteristic through clique counts in the partition graph. The second is a nerve-side formula arising from the canonical good cover by distinct full star- and full top-simplices, which yields
\[
\chi(K_n)=\chi(N_n),
\]
where $N_n$ is the corresponding nerve. We further use the classification of maximal simplices into star-, top-, and edge-type pieces to formulate a local-to-global counting framework based on local admissibility data and global deduplication.

The paper is primarily organizational and computational. It fixes a consistent counting dictionary, separates intrinsic global counts from auxiliary based counts, records exact data for the full main sequence package on $1\le n\le 25$, and extends the low-dimensional clique-count layer through $n=60$. We do not claim closed formulas for $\chi(K_n)$ or for the full family of clique counts. Rather, the paper provides a framework in which such questions can be studied systematically.
\end{abstract}



\section{Introduction}

For each $n\ge 1$, let $G_n$ be the partition graph whose vertices are the integer partitions of $n$, with adjacency given by an elementary transfer of one cell between two parts, followed by reordering. Let
\[
K_n:=\Cl(G_n)
\]
be the clique complex of $G_n$.

The qualitative topology of $K_n$ has already been determined in \cite{Lyudogovskiy2026Homotopy}: for every $n\ge 1$,
\[
K_n \simeq \bigvee^{\,b_n} S^2,
\qquad
b_n=\chi(K_n)-1.
\]
In the present paper we use the homotopy theorem, the good-cover/nerve equivalence, the containment of every clique in a full star- or full top-simplex, and the classification of maximal simplices from \cite{Lyudogovskiy2026Homotopy} as established inputs and do not reprove them here. Thus the remaining task is numerical: once the homotopy type is known, the topology of $K_n$ is encoded by explicit integer invariants, first of all by the Euler characteristic $\chi(K_n)$ and the derived sequence
\[
b_n=\chi(K_n)-1.
\]

This paper organizes that numerical layer systematically. For general background on integer partitions and Ferrers-diagram language we refer to Andrews \cite{Andrews1976}; for standard background on clique complexes, nerve constructions, and combinatorial-algebraic topology we refer to Bj{\"o}rner \cite{Bjorner1995} and Kozlov \cite{Kozlov2008}. We study the Euler characteristic of $K_n$, the clique counts
\[
c_r(n)=\#\{\text{cliques of size }r\text{ in }G_n\},
\]
the associated simplex counts, and several distinct families of maximal-simplex and cover-side counts. The aim is not to reprove the wedge-of-spheres theorem, but to turn it into a clean counting framework and a coherent family of integer sequences.

Two counting languages are available. The first is the direct clique-counting identity
\[
\chi(K_n)=\sum_{r\ge 1}(-1)^{r-1}c_r(n),
\]
which expresses Euler characteristic through clique counts in the partition graph. The second comes from the canonical good cover by distinct full star- and full top-simplices: if $N_n$ denotes the nerve of that cover, then
\[
K_n\simeq N_n,
\qquad
\chi(K_n)=\chi(N_n).
\]
Accordingly, one may study the same invariant either through direct clique enumeration or through overlap data in the canonical cover.

A second key input is the classification of maximal simplices established in \cite{Lyudogovskiy2026Homotopy}. The maximal simplices of $K_n$ are precisely the full star-simplices, the full top-simplices, and a residual family of maximal edges. This leads to three related but distinct counting layers: counts of cover elements in the canonical cover, counts of genuine maximal simplices, and auxiliary based counts attached to local presentations before global deduplication.

The paper is therefore primarily organizational and computational. More precisely, it does the following:

\begin{enumerate}
\item we fix a consistent dictionary relating clique counts, simplex counts, Euler characteristic, reduced Euler characteristic, and the sequence $b_n$;
\item we isolate two exact frameworks for $\chi(K_n)$: the clique-counting framework and the nerve-counting framework;
\item we formulate a local-to-global counting scheme for maximal simplices, based on local admissibility data and global deduplication;
\item we separate systematically the intrinsic global counts from the auxiliary based counts that arise in computation;
\item we record exact numerical data for the full main sequence package on $1\le n\le 25$, extend the low-dimensional clique-count layer through $n=60$, and fix normalization conventions for the main sequence family.
\end{enumerate}

A central methodological point is that several closely related counting problems must be kept distinct. These include the Euler characteristic $\chi(K_n)$, the shifted quantity $b_n=\chi(K_n)-1$, the clique counts $c_r(n)$, the simplex counts $f_p(K_n)$, the counts of distinct maximal simplices, the counts of cover elements in the canonical cover, and the auxiliary based counts obtained before deduplication. These quantities are related but not interchangeable, and one purpose of the paper is to make those relationships explicit and the notation stable.

The paper is organized as follows. Section~2 fixes the counting conventions and the basic topology-to-counting dictionary. Section~3 presents the two exact formulas for the Euler characteristic: the direct clique formula and the nerve formula. Section~4 develops the maximal-simplex bookkeeping framework, with explicit separation between cover counts, maximal-simplex counts, and based counts. Section~5 records the numerical data and normalization conventions for the main family of sequences, including an extension of the low-dimensional clique-count layer through $n=60$. The final section summarizes the numerical picture and formulates several open problems concerning exact formulas, structural reductions, and sequence behaviour.

\section{Preliminaries and counting conventions}

For $n\ge 1$, let $G_n$ be the partition graph on $\Par(n)$, and let
\[
K_n:=\Cl(G_n)
\]
be its clique complex.

The qualitative topology of $K_n$ has already been established in the previous paper \cite{Lyudogovskiy2026Homotopy}: for every $n\ge 1$,
\[
K_n \simeq \bigvee^{\,b_n} S^2,
\qquad
b_n=\chi(K_n)-1.
\]
The emphasis is therefore numerical rather than qualitative: we study the Euler characteristic of $K_n$, the associated simplex and clique counts, the counting of maximal simplices, and the resulting integer sequences.

Throughout, the qualitative results established in \cite{Lyudogovskiy2026Homotopy} are treated as black-box inputs. We recall only the notation needed for this counting-oriented study. Standard partition terminology is used throughout; see, for example, \cite{Andrews1976}. For the simplicial and nerve-theoretic background behind the clique-complex and good-cover language, see \cite{Bjorner1995,Kozlov2008}.

\subsection{The partition graph and its clique complex}

The vertex set of $G_n$ is the set $\Par(n)$ of all partitions of $n$. Two partitions are adjacent if one is obtained from the other by an elementary transfer of one cell between two parts, followed by reordering.

The simplices of $K_n=\Cl(G_n)$ are exactly the cliques of $G_n$. This leads to two parallel counting conventions.

\begin{definition}
For $p\ge 0$, let
\[
f_p(n):=f_p(K_n)
\]
denote the number of $p$-dimensional simplices of $K_n$.

For $r\ge 1$, let
\[
c_r(n)
\]
denote the number of cliques of size $r$ in $G_n$.
\end{definition}

Thus
\[
c_r(n)=f_{r-1}(n)\qquad (r\ge 1).
\]
We shall use $c_r(n)$ as the primary notation, since the graph-theoretic origin of the complex remains visible in it, and switch to $f_p(n)$ only when the simplicial grading is more convenient.

\subsection{Euler characteristic and the sequence $b_n$}

We write
\[
\chi_n:=\chi(K_n).
\]
Since $K_n$ is finite, its Euler characteristic is given by
\[
\chi_n=\sum_{p\ge 0}(-1)^p f_p(n).
\]
Equivalently,
\[
\chi_n=\sum_{r\ge 1}(-1)^{r-1} c_r(n).
\]

We also define
\[
b_n:=\chi_n-1.
\]
By the wedge-of-spheres theorem proved earlier, $b_n$ is at the same time the number of $2$-spheres in the wedge decomposition of $K_n$, the rank of $\widetilde H_2(K_n)$, and the reduced Euler characteristic $\widetilde\chi(K_n)$. Accordingly, the problems of determining $\chi_n$ and $b_n$ are equivalent.

\subsection{Canonical star/top simplices}

We use the terminology of admissible transfers, star-simplices, and top-simplices from \cite{Lyudogovskiy2026Homotopy}. In particular, for $\lambda\vdash n$, a removable corner $c$ of $\lambda$, and an addable corner $a$, the notation $\lambda(c\to a)$ denotes the partition obtained by an admissible transfer from $c$ to $a$ in the sense of \cite{Lyudogovskiy2026Homotopy}. We do not revisit the local classification here; in the present paper we need only the associated full star- and full top-simplices.

Let $\lambda\vdash n$, let $c$ be a removable corner of $\lambda$, and let $a$ be an addable corner. Define
\[
A_{\max}(\lambda,c):=\{\,a: \lambda(c\to a)\text{ is admissible}\,\},
\]
and
\[
C_{\max}(\lambda,a):=\{\,c: \lambda(c\to a)\text{ is admissible}\,\}.
\]

The corresponding full star- and full top-simplices are
\[
\Sigma^{\star}_{\max}(\lambda,c)
:=
\{\lambda\}\cup\{\lambda(c\to a):a\in A_{\max}(\lambda,c)\},
\]
and
\[
\Sigma^{\top}_{\max}(\lambda,a)
:=
\{\lambda\}\cup\{\lambda(c\to a):c\in C_{\max}(\lambda,a)\}.
\]

Every star-simplex is contained in a full star-simplex, and every top-simplex is contained in a full top-simplex. Moreover, every clique of $G_n$ is contained in one of these two canonical families \cite[Section~3]{Lyudogovskiy2026Homotopy}.

\subsection{The canonical cover and the nerve}

Let
\[
C_n
\]
denote the family of all distinct full star-simplices and full top-simplices in $K_n$. Then
\[
K_n=\bigcup_{U\in C_n} U.
\]

Let
\[
N_n:=N(C_n)
\]
be the nerve of this cover. The previous paper proves that the union of the full star- and full top-simplices equals $K_n$. Since each cover element is a simplex and any non-empty intersection of two such simplices is again a simplex (hence contractible), $C_n$ is a good cover. Consequently,
\[
K_n\simeq N_n.
\]
We use this cover-and-nerve theorem throughout as an established input from \cite{Lyudogovskiy2026Homotopy}.

Therefore $\chi(K_n)$ may be computed in two ways: either directly from the simplex counts of $K_n$, or from the simplex counts of the nerve $N_n$.

\subsection{Distinct counts and based counts}

A central bookkeeping issue is that several natural counting problems must be kept separate.

First, there are the global simplex counts
\[
f_p(n)\quad\text{or equivalently}\quad c_r(n).
\]

Second, there are the numbers of distinct maximal simplices of $K_n$. The previous paper shows that the maximal simplices of $K_n$ are precisely:

\begin{enumerate}
\item full star-simplices $\Sigma^{\star}_{\max}(\lambda,c)$ with $|A_{\max}(\lambda,c)|\ge 2$;
\item full top-simplices $\Sigma^{\top}_{\max}(\lambda,a)$ with $|C_{\max}(\lambda,a)|\ge 2$;
\item certain maximal edges arising in the one-dimensional residual case.
\end{enumerate}

We denote the corresponding numbers of distinct maximal simplices by
\[
m_{\star}(n),\qquad m_{\top}(n),\qquad m_e(n).
\]

Third, there are the associated based counts, in which one counts presentations before identifying coincident simplices. Namely, we may define
\[
m_{\star}^{\mathrm{based}}(n)
:=
\#\{(\lambda,c): |A_{\max}(\lambda,c)|\ge 2\},
\]
and
\[
m_{\top}^{\mathrm{based}}(n)
:=
\#\{(\lambda,a): |C_{\max}(\lambda,a)|\ge 2\}.
\]

These based counts are often easier to extract from local transfer data, but in general they need not coincide with the numbers of distinct maximal simplices: different pairs $(\lambda,c)$ or $(\lambda,a)$ may determine the same simplex. Whenever we use a based count rather than a distinct count, this will be stated explicitly.

\subsection{Basic counting identities}

\begin{proposition}[Topology-to-counting dictionary]
For every $n\ge 1$, one has
\[
c_r(n)=f_{r-1}(n)\qquad (r\ge 1),
\]
\[
\chi_n=\sum_{p\ge 0}(-1)^p f_p(n)
      =\sum_{r\ge 1}(-1)^{r-1}c_r(n),
\]
and
\[
\chi_n=\chi(N_n).
\]
Moreover,
\[
b_n=\chi_n-1.
\]
\end{proposition}

\begin{proof}
The identity $c_r(n)=f_{r-1}(n)$ is immediate from the definition of the clique complex. The Euler characteristic formulas are the standard ones for a finite simplicial complex. Finally, $K_n\simeq N_n$ by the good-cover theorem established earlier, so
\[
\chi(K_n)=\chi(N_n).
\]
The identity $b_n=\chi_n-1$ is the notation fixed above.
\end{proof}

\subsection{Conventions and scope}

Two indexing conventions coexist naturally in this subject:

\begin{itemize}
\item clique size $r$, leading to the sequences $c_r(n)$;
\item simplex dimension $p$, leading to the sequences $f_p(n)$.
\end{itemize}

In this paper the primary indexing will be by clique size. We also distinguish systematically between the Euler characteristic $\chi_n$, the reduced Euler characteristic $\widetilde\chi(K_n)=\chi_n-1$, the sequence $b_n=\chi_n-1$, distinct maximal-simplex counts, and based maximal-simplex counts.

In what follows, these conventions are used to state exact identities for $\chi_n$, to organize the maximal-simplex counting framework, and to present the computed sequence data.

\section{Euler characteristic via clique counts and the canonical nerve}

We now pass from the qualitative topology of $K_n$ to its numerical content. The decisive invariant is the Euler characteristic
\[
\chi_n:=\chi(K_n),
\]
because the previously established homotopy classification gives
\[
K_n \simeq \bigvee^{\,b_n} S^2,
\qquad
b_n=\chi_n-1.
\]
Thus every effective method for computing $\chi_n$ immediately yields the precise homotopy type of $K_n$.

In this section we record two complementary counting frameworks for $\chi_n$: the direct clique-counting framework on $G_n$, and the indirect nerve-counting framework coming from the canonical cover $C_n$.

\subsection{Euler characteristic from clique counts}

Since $K_n$ is a finite simplicial complex, its Euler characteristic is given by the alternating sum of its simplex numbers:
\[
\chi_n=\sum_{p\ge 0}(-1)^p f_p(n).
\]
Using the relation $c_r(n)=f_{r-1}(n)$, we obtain the equivalent clique-counting formula
\[
\chi_n=\sum_{r\ge 1}(-1)^{r-1}c_r(n).
\]

\begin{proposition}\label{prop:chi-clique-formula}
For every $n\ge 1$,
\[
\chi_n=\sum_{r\ge 1}(-1)^{r-1}c_r(n).
\]
Equivalently,
\[
b_n=\sum_{r\ge 1}(-1)^{r-1}c_r(n)-1.
\]
\end{proposition}

\begin{proof}
The first identity is just the Euler characteristic formula for the finite simplicial complex $K_n$, rewritten in clique language via $c_r(n)=f_{r-1}(n)$. The second is the definition $b_n=\chi_n-1$.
\end{proof}

Proposition~\ref{prop:chi-clique-formula} is basic but useful. It shows that the global homotopy type of $K_n$ is encoded by a single alternating sum of clique numbers in the partition graph $G_n$. Accordingly, any structural information on the sequences $c_r(n)$ feeds directly into the topology.

\subsection{Finite support and maximal clique size}

For each fixed $n$, the graph $G_n$ is finite, hence $K_n$ has only finitely many simplices. In particular, only finitely many clique numbers $c_r(n)$ are nonzero.

It is therefore convenient to introduce the maximal clique size
\[
\omega_n:=\omega(G_n)=\max\{\,r : c_r(n)\neq 0\,\}.
\]
Equivalently, $\omega_n-1$ is the dimension of the largest simplex of $K_n$.

Thus the Euler characteristic formula may be written more explicitly as
\[
\chi_n=\sum_{r=1}^{\omega_n}(-1)^{r-1}c_r(n).
\]

\begin{remark}
The quantity $\omega_n$ is not itself the main object of the present paper, but it is a useful bookkeeping parameter: it determines the length of the alternating sum defining $\chi_n$, and it separates the fixed-$r$ clique-counting regime from the maximal-dimension regime.
\end{remark}

\subsection{The nerve-side formula}

The second counting language comes from the canonical cover. Let $C_n$ be the canonical cover of $K_n$ by all distinct full star-simplices and full top-simplices, and let
\[
N_n:=N(C_n)
\]
be its nerve.

By the good-cover theorem established in the previous paper, one has
\[
K_n\simeq N_n.
\]
Since Euler characteristic is a homotopy invariant for finite complexes, it follows immediately that
\[
\chi_n=\chi(N_n).
\]

\begin{proposition}\label{prop:chi-nerve-formula}
For every $n\ge 1$,
\[
\chi_n=\chi(N_n).
\]
Equivalently,
\[
b_n=\chi(N_n)-1.
\]
\end{proposition}

\begin{proof}
The canonical cover $C_n$ is a good cover of $K_n$, hence $K_n\simeq N_n$. Therefore
\[
\chi(K_n)=\chi(N_n).
\]
The second identity again follows from $b_n=\chi_n-1$.
\end{proof}

This gives an alternative route to the same invariant: instead of counting cliques in $G_n$, one may count simplices in the nerve $N_n$, that is, nonempty intersections of members of the canonical cover.

\subsection{Two counting languages for the same invariant}

\begin{proposition}\label{prop:two-languages}
For every $n\ge 1$,
\[
\chi_n
=
\sum_{r=1}^{\omega_n}(-1)^{r-1}c_r(n)
=
\chi(N_n).
\]
\end{proposition}

\begin{proof}
Combine Propositions~\ref{prop:chi-clique-formula} and~\ref{prop:chi-nerve-formula}.
\end{proof}

The significance of Proposition~\ref{prop:two-languages} is methodological: the same invariant may be approached in two genuinely different ways. One proceeds on the level of all cliques in the graph $G_n$, while the other proceeds on the level of intersections in the canonical cover. The first language is closer to the original graph structure and leads naturally to the sequences $c_r(n)$. The second is closer to the canonical decomposition of clique geometry into star- and top-mechanisms, and is potentially better suited to computations based on maximal simplices and intersection data.

\subsection{Alternating cancellation and numerical organization}

Formula
\[
\chi_n=\sum_{r=1}^{\omega_n}(-1)^{r-1}c_r(n)
\]
shows that the Euler characteristic is produced by alternating cancellation among clique counts of different sizes; no single simplex dimension governs it in isolation.

This suggests three distinct numerical levels:

\begin{enumerate}
\item the raw clique counts $c_r(n)$;
\item the derived Euler characteristic $\chi_n$;
\item the shifted invariant $b_n=\chi_n-1$, which directly records the wedge multiplicity.
\end{enumerate}

From the viewpoint of integer sequences, these levels should be kept separate. The sequence $b_n$ is topologically primary, but it is not combinatorially primitive: it is derived from the full simplex-counting profile of $K_n$. Conversely, the sequences $c_r(n)$ are combinatorially more basic, but no single one of them determines the homotopy type on its own.

\subsection{Consequences for the rest of the paper}

The remaining sections pursue these two counting frameworks in parallel.

On the clique side, we study the families $c_r(n)$, total simplex counts, and maximal simplex counts. On the cover side, we use the canonical star/top description to organize maximal simplices and to compare distinct counts with based counts arising from local transfer data.

The computational part of the paper is therefore an integral part of the counting framework. Its purpose is to make the identities above effective and to extract from them clean numerical sequences suitable for further study.

\section{Maximal simplices and counting frameworks}

The Euler characteristic identities of Section~3 reduce the topology of $K_n$ to counting. A direct enumeration of all cliques in $G_n$ is one exact route, but it is not the only one. The structural results proved in \cite{Lyudogovskiy2026Homotopy} provide a second route: one may organize simplex counting through the canonical star/top mechanisms and the corresponding families of full simplices.

This section organizes the relevant counting problems. Several closely related layers occur here, and they must be separated carefully: cover elements versus maximal simplices, distinct counts versus based counts, and intrinsic clique counts versus auxiliary local presentation counts.

\subsection{Maximal simplices and cover elements}

Recall that for $\lambda\vdash n$, a removable corner $c$, and an addable corner $a$, the full star- and full top-simplices are
\[
\Sigma^{\star}_{\max}(\lambda,c)
=
\{\lambda\}\cup\{\lambda(c\to b): b\in A_{\max}(\lambda,c)\},
\]
and
\[
\Sigma^{\top}_{\max}(\lambda,a)
=
\{\lambda\}\cup\{\lambda(d\to a): d\in C_{\max}(\lambda,a)\}.
\]
These are exactly the full simplices used in the canonical cover
\[
C_n
\]
of $K_n$.

The earlier paper proves that for $n\ge 2$ the maximal simplices of $K_n$ are precisely:
\begin{enumerate}
\item the full star-simplices $\Sigma^{\star}_{\max}(\lambda,c)$ with $|A_{\max}(\lambda,c)|\ge 2$;
\item the full top-simplices $\Sigma^{\top}_{\max}(\lambda,a)$ with $|C_{\max}(\lambda,a)|\ge 2$;
\item the edges $\{\lambda,\lambda(c\to a)\}$ such that
\[
|A_{\max}(\lambda,c)|=1,
\qquad
|C_{\max}(\lambda,a)|=1.
\]
\end{enumerate}
The case $n=1$ is trivial: $K_1$ consists of a single maximal vertex. For $n\ge 2$, every simplex of $K_n$ is contained in a full star-simplex or in a full top-simplex, and hence in a maximal simplex \cite[Lemma~3.13]{Lyudogovskiy2026Homotopy}.

Thus two distinct families appear naturally:
\begin{itemize}
\item the family $C_n$ of all distinct full star- and full top-simplices, which serves as the cover for the nerve construction;
\item the family of genuine maximal simplices of $K_n$, consisting of the maximal full star-simplices, the maximal full top-simplices, and the residual maximal edges.
\end{itemize}
These families are related but not identical.

\subsection{Distinct global counts}

We now attach numerical invariants to these families.

For the canonical cover, write
\[
\begin{aligned}
\nu_{\star}(n)&:=\#\{\text{distinct full star-simplices in }C_n\},\\
\nu_{\top}(n)&:=\#\{\text{distinct full top-simplices in }C_n\},
\end{aligned}
\]
and define
\[
\nu_C(n):=|C_n|.
\]
Thus $\nu_C(n)$ is the number of vertices of the nerve $N_n=N(C_n)$.

In general, $\nu_C(n)$ need not equal $\nu_{\star}(n)+\nu_{\top}(n)$. A full simplex may occur simultaneously as a full star-simplex and as a full top-simplex, so the cover level must be counted after taking distinct vertex sets. For $n=1$ we use the convention that the unique vertex of $K_1$ is counted as a single coincident full star-simplex and full top-simplex; accordingly,
\[
\nu_{\star}(1)=\nu_{\top}(1)=\nu_C(1)=1.
\]
At the maximal-simplex level one has
\[
m_{\star}(1)=m_{\top}(1)=m_e(1)=0,
\qquad
m_{\max}(1)=1.
\]

For genuine maximal simplices, we write
\[
m_{\star}(n),\qquad m_{\top}(n),\qquad m_e(n)
\]
for the numbers of distinct maximal simplices of star-type, top-type, and edge-type, respectively. If needed, we also write
\[
m_{\max}(n)
\]
for the total number of distinct maximal simplices of $K_n$, computed after global deduplication across the union of all maximal star-, top-, and edge-presentations. In particular,
\[
m_{\max}(1)=1,
\]
because $K_1$ is a single vertex.

\begin{remark}
The counts $\nu_{\star}(n),\nu_{\top}(n),\nu_C(n)$ and the counts $m_{\star}(n),m_{\top}(n),m_e(n)$ live on different levels. The former belong to the cover-and-nerve language; the latter belong to the intrinsic maximal-simplex structure of $K_n$. In particular, residual maximal edges contribute to $m_e(n)$ but not to $\nu_C(n)$, while a full star- or full top-simplex may contribute to $\nu_C(n)$ without being maximal in the whole complex.
\end{remark}

This distinction is important numerically: cover counts, maximal-simplex counts, and clique counts should not be conflated.

\subsection{Based presentations and deduplication}

For actual computation it is often easier to count local presentations first and only then pass to distinct global objects.

Define the based star and top counts by
\[
m_{\star}^{\mathrm{based}}(n)
:=
\#\{(\lambda,c): \lambda\vdash n,\ c\text{ removable in }\lambda,\ |A_{\max}(\lambda,c)|\ge 2\},
\]
and
\[
m_{\top}^{\mathrm{based}}(n)
:=
\#\{(\lambda,a): \lambda\vdash n,\ a\text{ addable in }\lambda,\ |C_{\max}(\lambda,a)|\ge 2\}.
\]
For the residual one-dimensional case, define the based edge count by
\[
m_{e}^{\mathrm{based}}(n)
:=
\#\{(\lambda,c,a): \lambda(c\to a)\text{ admissible},\ |A_{\max}(\lambda,c)|=1,\ |C_{\max}(\lambda,a)|=1\}.
\]
Here the counting convention is oriented: a presentation $(\lambda,c,a)$ and the reversed presentation based at $\lambda(c\to a)$ are counted separately whenever both occur.

The corresponding cover-element counts also admit based versions if needed, namely by counting pairs $(\lambda,c)$ and $(\lambda,a)$ without imposing maximality. We shall not introduce separate notation for them here, because the main bookkeeping issue in the present paper concerns maximal simplices.

\begin{remark}
Based counts need not coincide with distinct counts. Different pairs $(\lambda,c)$ may generate the same full star-simplex, and different pairs $(\lambda,a)$ may generate the same full top-simplex. Likewise, a single undirected maximal edge may have two oriented based presentations. Hence the passage from based counts to distinct counts is a genuine deduplication step.
\end{remark}

Accordingly, based counts are local-combinatorial quantities, whereas the distinct counts $\nu_{\star}(n),\nu_{\top}(n),m_{\star}(n),m_{\top}(n),m_e(n)$ are global-combinatorial quantities.

\subsection{Local-to-global counting of cliques}

The same distinction persists for arbitrary simplices, not only for maximal ones. Since every simplex of $K_n$ is contained in a full star-simplex or a full top-simplex, every clique counted by $c_r(n)$ arises inside at least one canonical container.

This suggests two auxiliary counting layers for each $r\ge 1$:
\[
c_{r}^{\star,\mathrm{based}}(n),\qquad c_{r}^{\top,\mathrm{based}}(n),
\]
where an $r$-clique is counted together with a chosen star-base $(\lambda,c)$ or top-base $(\lambda,a)$. These are local presentation counts, not intrinsic invariants.

The intrinsic quantity remains
\[
c_r(n),
\]
the number of distinct $r$-cliques in $G_n$. Passing from the based counts to $c_r(n)$ again requires global identification of equal vertex sets.

\begin{proposition}\label{prop:counting-framework}
For each $n\ge 1$, the simplex counts of $K_n$ can be recovered by a finite local-to-global counting scheme based on the admissibility sets $A_{\max}(\lambda,c)$ and $C_{\max}(\lambda,a)$.

More precisely:
\begin{enumerate}
\item the based maximal counts
\[
m_{\star}^{\mathrm{based}}(n),\quad m_{\top}^{\mathrm{based}}(n),\quad m_e^{\mathrm{based}}(n)
\]
are determined directly by the local admissibility data;
\item the distinct maximal counts
\[
m_{\star}(n),\quad m_{\top}(n),\quad m_e(n)
\]
are obtained from the based presentations by deduplication of coincident vertex sets;
\item for every $r\ge 1$, the clique count $c_r(n)$ is obtained by enumerating all $r$-vertex subsimplices of the full star- and full top-simplices and then deduplicating globally.
\end{enumerate}
\end{proposition}

\begin{proof}
The first statement follows directly from the definitions. For the second, each distinct maximal simplex of star-, top-, or edge-type is represented by at least one based presentation, and two presentations determine the same distinct maximal simplex exactly when they yield the same vertex set. For the third, the clique-classification results recalled from \cite{Lyudogovskiy2026Homotopy} imply that every simplex of $K_n$ lies in a full star-simplex or a full top-simplex. Hence every $r$-clique appears as an $r$-vertex subsimplex of at least one such full simplex, and after quotienting by equality of vertex sets one obtains exactly the set of distinct $r$-cliques of $G_n$.
\end{proof}

Proposition~\ref{prop:counting-framework} is algorithmic rather than closed-form. Its role is to isolate the interface between local data and global enumeration.

\subsection{Consequences for sequence layers}

The discussion above singles out four sequence layers.

\begin{enumerate}
\item \emph{Topological sequences:}
\[
\chi_n,\qquad b_n.
\]

\item \emph{Intrinsic clique and maximal-simplex sequences:}
\[
c_r(n),\qquad m_{\star}(n),\qquad m_{\top}(n),\qquad m_e(n).
\]

\item \emph{Cover sequences:}
\[
\nu_{\star}(n),\qquad \nu_{\top}(n),\qquad \nu_C(n).
\]
These belong to the nerve-side language.

\item \emph{Auxiliary based sequences:}
\[
m_{\star}^{\mathrm{based}}(n),\qquad m_{\top}^{\mathrm{based}}(n),\qquad m_e^{\mathrm{based}}(n),
\]
and, if desired,
\[
c_{r}^{\star,\mathrm{based}}(n),\qquad c_{r}^{\top,\mathrm{based}}(n).
\]
These describe the pre-deduplication layer of the computation.
\end{enumerate}

Only the first three layers consist of distinct global invariants. The fourth is algorithmically useful, but auxiliary. We maintain this separation throughout the computational section.

\section{Computational data and sequence organization}

The identities established in Sections~2--4 are exact for all $n$. We now record the corresponding numerical data. The tables below give exact values for the full numerical package in the range $1\le n\le 25$ and extend the low-dimensional clique counts to $1\le n\le 60$ in order to track changes in the dominant clique size. They also indicate which parts of the numerical layer come directly from global clique enumeration and which come from local star/top bookkeeping followed by deduplication.

Throughout, we distinguish carefully between exact formulas valid for all $n$ and values obtained by finite computation over the ranges $1\le n\le 25$ and $1\le n\le 60$, as specified in each table.

\subsection{Primary topological data}

The two most fundamental numerical sequences are
\[
\chi_n:=\chi(K_n)
\qquad\text{and}\qquad
b_n:=\chi_n-1.
\]
They are topologically primary because $b_n$ is exactly the wedge multiplicity in
\[
K_n \simeq \bigvee^{\,b_n} S^2.
\]

Table~\ref{tab:chi-bn-25} records the exact values of $\chi_n$ and $b_n$ for $1\le n\le 25$. In particular, $K_1$ consists of a single vertex and is contractible. For $2\le n\le 7$ one has $b_n=0$, so $K_n$ is contractible as well (equivalently, $\bigvee^0 S^2$ is interpreted as a point). The first nontrivial case is
\[
K_8\simeq S^2.
\]

\begin{table}[ht]
\centering
\small
\renewcommand{\arraystretch}{1.08}
\setlength{\tabcolsep}{6pt}
\begin{tabular}{rrr}
\toprule
$n$ & $\chi_n$ & $b_n$ \\
\midrule
1 & 1 & 0 \\
2 & 1 & 0 \\
3 & 1 & 0 \\
4 & 1 & 0 \\
5 & 1 & 0 \\
6 & 1 & 0 \\
7 & 1 & 0 \\
8 & 2 & 1 \\
9 & 3 & 2 \\
10 & 6 & 5 \\
11 & 11 & 10 \\
12 & 20 & 19 \\
13 & 33 & 32 \\
14 & 56 & 55 \\
15 & 88 & 87 \\
16 & 138 & 137 \\
17 & 208 & 207 \\
18 & 311 & 310 \\
19 & 452 & 451 \\
20 & 653 & 652 \\
21 & 922 & 921 \\
22 & 1294 & 1293 \\
23 & 1788 & 1787 \\
24 & 2454 & 2453 \\
25 & 3325 & 3324 \\
\bottomrule
\end{tabular}
\caption{Exact values of $\chi_n=\chi(K_n)$ and $b_n=\chi_n-1$ for $1\le n\le 25$.}
\label{tab:chi-bn-25}
\end{table}

\subsection{Fixed-$r$ clique-count sequences}

Beyond $\chi_n$ and $b_n$, the next natural layer consists of the clique-count sequences
\[
c_r(n)\qquad (r\ge 1),
\]
where $c_r(n)$ is the number of cliques of size $r$ in $G_n$, equivalently the number of $(r-1)$-simplices in $K_n$.

Table~\ref{tab:clique-counts-1-25} records the full clique profiles in the computed range. The first cases have the concrete interpretations
\[
c_1(n)=p(n),
\qquad
c_2(n)=\#\{\text{edges of }G_n\},
\qquad
c_3(n)=\#\{\text{triangles of }G_n\},
\]
and so on.

The data make visible the first occurrences of higher-dimensional simplices:
\[
c_4(n)>0\ \text{first at }n=7,\qquad
c_5(n)>0\ \text{first at }n=11,
\]
\[
c_6(n)>0\ \text{first at }n=16,\qquad
c_7(n)>0\ \text{first at }n=22.
\]
Equivalently, the maximal clique size $\omega_n$ increases at
\[
n=2,4,7,11,16,22
\]
within the computed range.

The Euler characteristic is recovered from the full clique profile by
\[
\chi_n=\sum_{r\ge 1}(-1)^{r-1}c_r(n),
\]
and the last column of Table~\ref{tab:clique-counts-1-25} records the resulting values of $\chi_n$.

\setlength{\LTleft}{\fill}
\setlength{\LTright}{\fill}
\begin{longtable}{rrrrrrrrrr}
\caption{Exact clique counts for $1\le n\le 25$. For each $n$, the table records the full clique profile $(c_r(n))_{r\ge 1}$ in the computed range, together with the maximal clique size $\omega_n$ and the resulting Euler characteristic $\chi_n=\sum_{r\ge 1}(-1)^{r-1}c_r(n)$. Blank clique sizes do not occur; they are recorded here as $0$.}\label{tab:clique-counts-1-25}\\
\toprule
$n$ & $c_1$ & $c_2$ & $c_3$ & $c_4$ & $c_5$ & $c_6$ & $c_7$ & $\omega_n$ & $\chi_n$ \\
\midrule
\endfirsthead
\multicolumn{10}{c}{\tablename\ \thetable\ -- continued from previous page}\\
\toprule
$n$ & $c_1$ & $c_2$ & $c_3$ & $c_4$ & $c_5$ & $c_6$ & $c_7$ & $\omega_n$ & $\chi_n$ \\
\midrule
\endhead
\midrule
\multicolumn{10}{r}{continued on next page}\\
\endfoot
\bottomrule
\endlastfoot
1 & 1 & 0 & 0 & 0 & 0 & 0 & 0 & 1 & 1 \\
2 & 2 & 1 & 0 & 0 & 0 & 0 & 0 & 2 & 1 \\
3 & 3 & 2 & 0 & 0 & 0 & 0 & 0 & 2 & 1 \\
4 & 5 & 5 & 1 & 0 & 0 & 0 & 0 & 3 & 1 \\
5 & 7 & 9 & 3 & 0 & 0 & 0 & 0 & 3 & 1 \\
6 & 11 & 17 & 7 & 0 & 0 & 0 & 0 & 3 & 1 \\
7 & 15 & 28 & 15 & 1 & 0 & 0 & 0 & 4 & 1 \\
8 & 22 & 47 & 29 & 2 & 0 & 0 & 0 & 4 & 2 \\
9 & 30 & 73 & 52 & 6 & 0 & 0 & 0 & 4 & 3 \\
10 & 42 & 114 & 90 & 12 & 0 & 0 & 0 & 4 & 6 \\
11 & 56 & 170 & 149 & 25 & 1 & 0 & 0 & 5 & 11 \\
12 & 77 & 253 & 239 & 45 & 2 & 0 & 0 & 5 & 20 \\
13 & 101 & 365 & 374 & 82 & 5 & 0 & 0 & 5 & 33 \\
14 & 135 & 525 & 572 & 137 & 11 & 0 & 0 & 5 & 56 \\
15 & 176 & 738 & 857 & 229 & 22 & 0 & 0 & 5 & 88 \\
16 & 231 & 1033 & 1264 & 364 & 41 & 1 & 0 & 6 & 138 \\
17 & 297 & 1422 & 1835 & 574 & 74 & 2 & 0 & 6 & 208 \\
18 & 385 & 1948 & 2628 & 876 & 127 & 5 & 0 & 6 & 311 \\
19 & 490 & 2634 & 3719 & 1326 & 213 & 10 & 0 & 6 & 452 \\
20 & 627 & 3545 & 5205 & 1959 & 346 & 21 & 0 & 6 & 653 \\
21 & 792 & 4721 & 7210 & 2871 & 550 & 38 & 0 & 6 & 922 \\
22 & 1002 & 6259 & 9898 & 4133 & 855 & 70 & 1 & 7 & 1294 \\
23 & 1255 & 8227 & 13471 & 5902 & 1308 & 119 & 2 & 7 & 1788 \\
24 & 1575 & 10767 & 18190 & 8312 & 1965 & 202 & 5 & 7 & 2454 \\
25 & 1958 & 13990 & 24384 & 11621 & 2912 & 328 & 10 & 7 & 3325 \\
\end{longtable}

Define the low-dimensional leader profile by
\[
L_{\le 5}(n):=\{\,r\in\{1,\dots,5\}: c_r(n)=\max_{1\le s\le 5}c_s(n)\,\}.
\]
Table~\ref{tab:clique-counts-1-60-summary} extends the low-dimensional clique counts through $n=60$ and marks in bold the rowwise maxima among the displayed values $c_1(n),\dots,c_5(n)$.

\setlength{\LTleft}{\fill}
\setlength{\LTright}{\fill}
\begin{longtable}{rrrrrrrr}
\caption{Low-dimensional clique counts for $1\le n\le 60$. The entries in bold mark the rowwise maxima among the displayed counts $c_1(n),\dots,c_5(n)$. The final column records the corresponding leader profile $L_{\le 5}(n)$.}\label{tab:clique-counts-1-60-summary}\\
\toprule
$n$ & $c_1$ & $c_2$ & $c_3$ & $c_4$ & $c_5$ & $\omega_n$ & $L_{\le 5}(n)$ \\
\midrule
\endfirsthead
\multicolumn{8}{c}{\tablename\ \thetable\ -- continued from previous page}\\
\toprule
$n$ & $c_1$ & $c_2$ & $c_3$ & $c_4$ & $c_5$ & $\omega_n$ & $L_{\le 5}(n)$ \\
\midrule
\endhead
\midrule
\multicolumn{8}{r}{continued on next page}\\
\endfoot
\bottomrule
\endlastfoot
1 & \textbf{1} & 0 & 0 & 0 & 0 & 1 & $\{1\}$ \\
2 & \textbf{2} & 1 & 0 & 0 & 0 & 2 & $\{1\}$ \\
3 & \textbf{3} & 2 & 0 & 0 & 0 & 2 & $\{1\}$ \\
4 & \textbf{5} & \textbf{5} & 1 & 0 & 0 & 3 & $\{1,2\}$ \\
5 & 7 & \textbf{9} & 3 & 0 & 0 & 3 & $\{2\}$ \\
6 & 11 & \textbf{17} & 7 & 0 & 0 & 3 & $\{2\}$ \\
7 & 15 & \textbf{28} & 15 & 1 & 0 & 4 & $\{2\}$ \\
8 & 22 & \textbf{47} & 29 & 2 & 0 & 4 & $\{2\}$ \\
9 & 30 & \textbf{73} & 52 & 6 & 0 & 4 & $\{2\}$ \\
10 & 42 & \textbf{114} & 90 & 12 & 0 & 4 & $\{2\}$ \\
11 & 56 & \textbf{170} & 149 & 25 & 1 & 5 & $\{2\}$ \\
12 & 77 & \textbf{253} & 239 & 45 & 2 & 5 & $\{2\}$ \\
13 & 101 & 365 & \textbf{374} & 82 & 5 & 5 & $\{3\}$ \\
14 & 135 & 525 & \textbf{572} & 137 & 11 & 5 & $\{3\}$ \\
15 & 176 & 738 & \textbf{857} & 229 & 22 & 5 & $\{3\}$ \\
16 & 231 & 1033 & \textbf{1264} & 364 & 41 & 6 & $\{3\}$ \\
17 & 297 & 1422 & \textbf{1835} & 574 & 74 & 6 & $\{3\}$ \\
18 & 385 & 1948 & \textbf{2628} & 876 & 127 & 6 & $\{3\}$ \\
19 & 490 & 2634 & \textbf{3719} & 1326 & 213 & 6 & $\{3\}$ \\
20 & 627 & 3545 & \textbf{5205} & 1959 & 346 & 6 & $\{3\}$ \\
21 & 792 & 4721 & \textbf{7210} & 2871 & 550 & 6 & $\{3\}$ \\
22 & 1002 & 6259 & \textbf{9898} & 4133 & 855 & 7 & $\{3\}$ \\
23 & 1255 & 8227 & \textbf{13471} & 5902 & 1308 & 7 & $\{3\}$ \\
24 & 1575 & 10767 & \textbf{18190} & 8312 & 1965 & 7 & $\{3\}$ \\
25 & 1958 & 13990 & \textbf{24384} & 11621 & 2912 & 7 & $\{3\}$ \\
26 & 2436 & 18105 & \textbf{32465} & 16064 & 4257 & 7 & $\{3\}$ \\
27 & 3010 & 23286 & \textbf{42947} & 22059 & 6150 & 7 & $\{3\}$ \\
28 & 3718 & 29837 & \textbf{56479} & 30010 & 8784 & 7 & $\{3\}$ \\
29 & 4565 & 38028 & \textbf{73856} & 40578 & 12424 & 8 & $\{3\}$ \\
30 & 5604 & 48297 & \textbf{96070} & 54438 & 17402 & 8 & $\{3\}$ \\
31 & 6842 & 61053 & \textbf{124344} & 72629 & 24169 & 8 & $\{3\}$ \\
32 & 8349 & 76926 & \textbf{160181} & 96243 & 33291 & 8 & $\{3\}$ \\
33 & 10143 & 96524 & \textbf{205420} & 126894 & 45514 & 8 & $\{3\}$ \\
34 & 12310 & 120746 & \textbf{262323} & 166322 & 61780 & 8 & $\{3\}$ \\
35 & 14883 & 150487 & \textbf{333631} & 216996 & 83311 & 8 & $\{3\}$ \\
36 & 17977 & 187019 & \textbf{422690} & 281640 & 111634 & 8 & $\{3\}$ \\
37 & 21637 & 231643 & \textbf{533556} & 364001 & 148714 & 9 & $\{3\}$ \\
38 & 26015 & 286152 & \textbf{671137} & 468266 & 197000 & 9 & $\{3\}$ \\
39 & 31185 & 352413 & \textbf{841352} & 600065 & 259596 & 9 & $\{3\}$ \\
40 & 37338 & 432937 & \textbf{1051351} & 765758 & 340359 & 9 & $\{3\}$ \\
41 & 44583 & 530383 & \textbf{1309702} & 973710 & 444141 & 9 & $\{3\}$ \\
42 & 53174 & 648245 & \textbf{1626702} & 1233465 & 576928 & 9 & $\{3\}$ \\
43 & 63261 & 790274 & \textbf{2014671} & 1557354 & 746198 & 9 & $\{3\}$ \\
44 & 75175 & 961310 & \textbf{2488330} & 1959523 & 961141 & 9 & $\{3\}$ \\
45 & 89134 & 1166600 & \textbf{3065221} & 2457998 & 1233137 & 9 & $\{3\}$ \\
46 & 105558 & 1412811 & \textbf{3766255} & 3073550 & 1576126 & 10 & $\{3\}$ \\
47 & 124754 & 1707235 & \textbf{4616240} & 3832301 & 2007262 & 10 & $\{3\}$ \\
48 & 147273 & 2059004 & \textbf{5644632} & 4764468 & 2547446 & 10 & $\{3\}$ \\
49 & 173525 & 2478182 & \textbf{6886301} & 5907638 & 3222270 & 10 & $\{3\}$ \\
50 & 204226 & 2977224 & \textbf{8382479} & 7305371 & 4062793 & 10 & $\{3\}$ \\
51 & 239943 & 3569927 & \textbf{10181818} & 9011367 & 5106821 & 10 & $\{3\}$ \\
52 & 281589 & 4273195 & \textbf{12341707} & 11087943 & 6400072 & 10 & $\{3\}$ \\
53 & 329931 & 5105841 & \textbf{14929630} & 13611274 & 7997939 & 10 & $\{3\}$ \\
54 & 386155 & 6090698 & \textbf{18024940} & 16669854 & 9967132 & 10 & $\{3\}$ \\
55 & 451276 & 7253275 & \textbf{21720743} & 20371070 & 12388135 & 10 & $\{3\}$ \\
56 & 526823 & 8624287 & \textbf{26126169} & 24839688 & 15357540 & 11 & $\{3\}$ \\
57 & 614154 & 10238140 & \textbf{31368913} & 30226176 & 18991375 & 11 & $\{3\}$ \\
58 & 715220 & 12135975 & \textbf{37598305} & 36705325 & 23428441 & 11 & $\{3\}$ \\
59 & 831820 & 14363982 & \textbf{44988596} & 44486708 & 28834842 & 11 & $\{3\}$ \\
60 & 966467 & 16977037 & 53743028 & \textbf{53813706} & 35408595 & 11 & $\{4\}$ \\
\end{longtable}

In the computed low-dimensional clique data, the dominant displayed clique size follows the pattern
\[
L_{\le 5}(n)=\{1\}\quad (1\le n\le 3),\qquad L_{\le 5}(4)=\{1,2\},
\]
\[
L_{\le 5}(n)=\{2\}\quad (5\le n\le 12),\qquad L_{\le 5}(n)=\{3\}\quad (13\le n\le 54).
\]
The extended low-dimensional data continue the triangle-dominant regime through $n=59$, but at $n=60$ the inequality reverses:
\[
c_3(60)=53743028,\qquad c_4(60)=53813706.
\]
Moreover, the higher clique counts at $n=60$ are
\[
c_5(60)=35408595,\quad c_6(60)=15387845,\quad c_7(60)=4318590,
\]
\[
c_8(60)=739742,\quad c_9(60)=68940,\quad c_{10}(60)=2712,\quad c_{11}(60)=20,
\]
Thus $c_4$ overtakes $c_3$ for the first time at $n=60$; moreover, at that endpoint tetrahedra are the global leader among clique sizes.

Figure~\ref{fig:clique-log-all} shows the growth of the low-dimensional clique counts on a logarithmic scale, while Figure~\ref{fig:c3-c4-crossover} isolates the first triangle--tetrahedron crossover.

\begin{figure}[ht]
\centering
\includegraphics[width=0.92\textwidth]{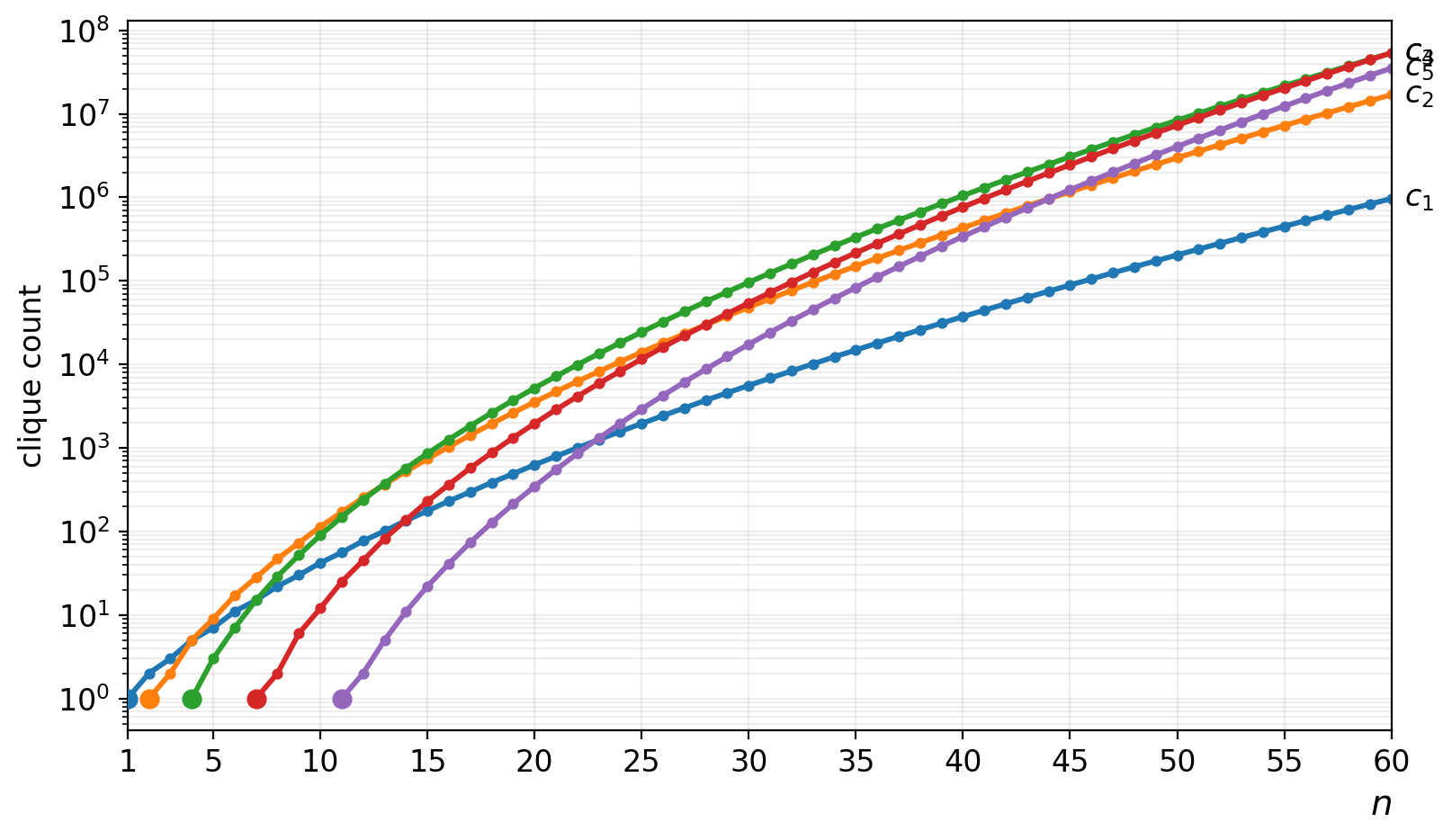}
\caption{Logarithmic plot of the low-dimensional clique counts $c_1(n),\dots,c_5(n)$ for $1\le n\le 60$. Each curve is drawn starting from its first positive value, and the first positive point is marked explicitly, so the logarithmic scale does not create artificial vertical segments from initial zeros.}
\label{fig:clique-log-all}
\end{figure}

\begin{figure}[ht]
\centering
\includegraphics[width=0.88\textwidth]{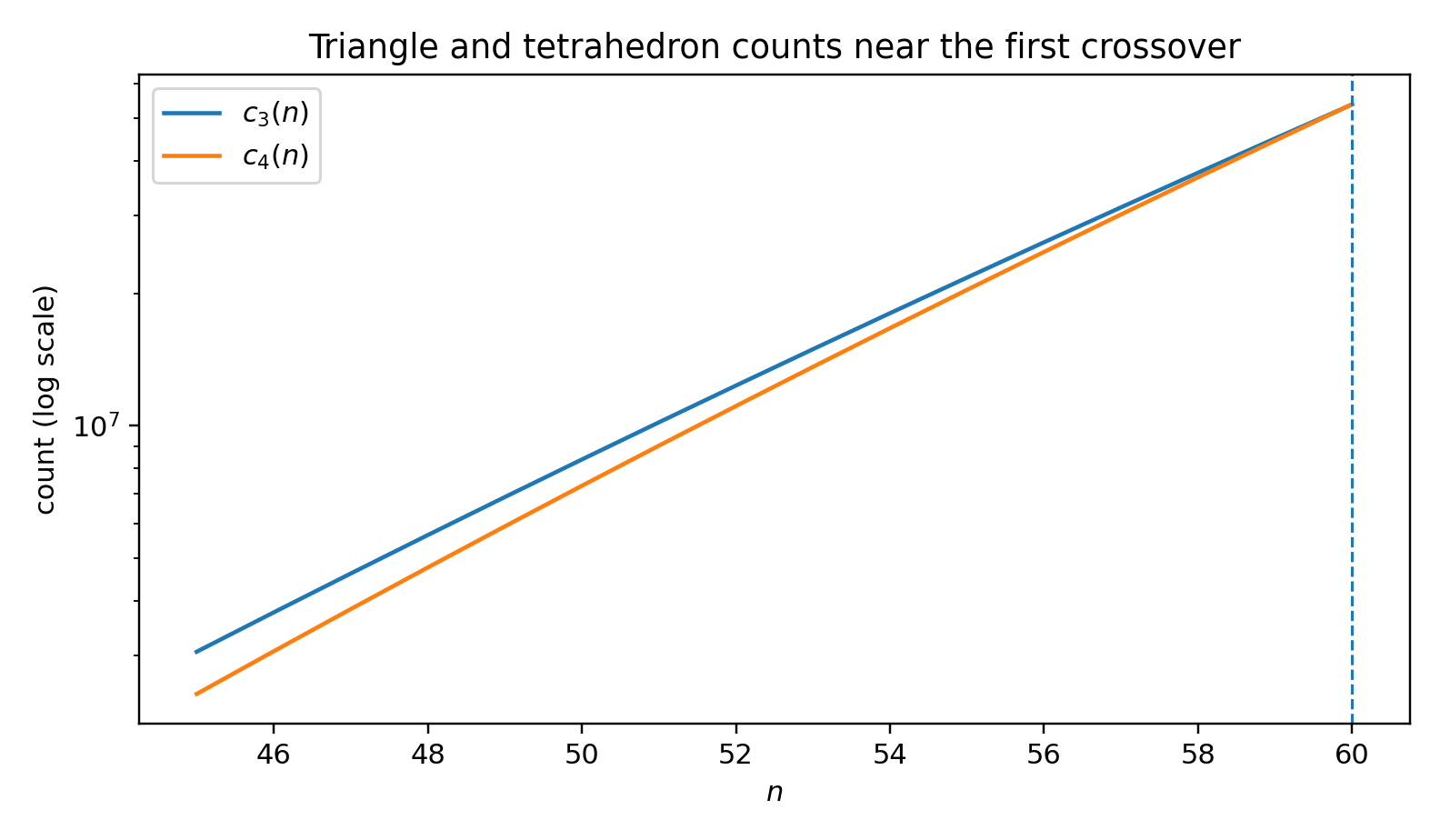}
\caption{Logarithmic plot of $c_3(n)$ and $c_4(n)$ for $45\le n\le 60$, showing the first crossover at $n=60$.}
\label{fig:c3-c4-crossover}
\end{figure}

\subsection{Maximal-simplex and cover data}

A second numerical layer is formed by the counts attached to the star/top container structure. On the cover side, the distinct counts are
\[
\nu_{\star}(n),\qquad \nu_{\top}(n),\qquad \nu_C(n)=|C_n|,
\]
where $\nu_C(n)$ is the number of vertices of the nerve $N_n$.

On the intrinsic maximal-simplex side, the basic counts are
\[
m_{\star}(n),\qquad m_{\top}(n),\qquad m_e(n),
\]
and, if desired, the total count
\[
m_{\max}(n)
\]
of distinct maximal simplices of $K_n$.

Table~\ref{tab:cover-max-counts} records these distinct counts for $1\le n\le 25$. Three points deserve emphasis here. First, $\nu_C(n)$ must be computed after deduplication of coincident full star- and full top-simplices; it is not a priori equal to $\nu_{\star}(n)+\nu_{\top}(n)$. Second, $m_{\max}(n)$ is computed independently as the number of distinct maximal simplices obtained after deduplicating the union of all maximal star-, top-, and edge-presentations. In the computed range, the resulting values agree with $m_{\star}(n)+m_{\top}(n)+m_e(n)$. Third, the residual maximal-edge count is observed to satisfy
\[
m_e(2)=1,
\qquad
m_e(n)=2 \quad (3\le n\le 25)
\]
in the present data set; whether this pattern continues for all $n\ge 3$ is left open here.

\begin{table}[ht]
\centering
\footnotesize
\renewcommand{\arraystretch}{1.08}
\setlength{\tabcolsep}{4pt}
\begin{tabular}{rrrrrrrr}
\toprule
$n$ & $\nu_{\star}$ & $\nu_{\top}$ & $\nu_C$ & $m_{\star}$ & $m_{\top}$ & $m_e$ & $m_{\max}$ \\
\midrule
1 & 1 & 1 & 1 & 0 & 0 & 0 & 1 \\
2 & 1 & 3 & 3 & 0 & 0 & 1 & 1 \\
3 & 2 & 5 & 5 & 0 & 0 & 2 & 2 \\
4 & 3 & 7 & 8 & 1 & 0 & 2 & 3 \\
5 & 5 & 11 & 14 & 2 & 1 & 2 & 5 \\
6 & 7 & 15 & 20 & 5 & 2 & 2 & 9 \\
7 & 11 & 22 & 31 & 7 & 5 & 2 & 14 \\
8 & 15 & 30 & 43 & 13 & 10 & 2 & 25 \\
9 & 22 & 42 & 62 & 18 & 16 & 2 & 36 \\
10 & 30 & 56 & 84 & 27 & 27 & 2 & 56 \\
11 & 42 & 77 & 117 & 38 & 42 & 2 & 82 \\
12 & 56 & 101 & 155 & 54 & 62 & 2 & 118 \\
13 & 77 & 135 & 210 & 71 & 87 & 2 & 160 \\
14 & 101 & 176 & 275 & 99 & 128 & 2 & 229 \\
15 & 135 & 231 & 364 & 131 & 171 & 2 & 304 \\
16 & 176 & 297 & 471 & 172 & 236 & 2 & 410 \\
17 & 231 & 385 & 614 & 226 & 311 & 2 & 539 \\
18 & 297 & 490 & 785 & 295 & 417 & 2 & 714 \\
19 & 385 & 627 & 1010 & 379 & 540 & 2 & 921 \\
20 & 490 & 792 & 1280 & 488 & 706 & 2 & 1196 \\
21 & 627 & 1002 & 1627 & 621 & 896 & 2 & 1519 \\
22 & 792 & 1255 & 2045 & 788 & 1156 & 2 & 1946 \\
23 & 1002 & 1575 & 2575 & 998 & 1455 & 2 & 2455 \\
24 & 1255 & 1958 & 3211 & 1253 & 1846 & 2 & 3101 \\
25 & 1575 & 2436 & 4009 & 1567 & 2296 & 2 & 3865 \\
\bottomrule
\end{tabular}
\caption{Distinct cover and maximal-simplex counts for $1\le n\le 25$. Here $\nu_C(n)=|C_n|$ is the number of distinct full star/top cover elements, while $m_{\max}(n)$ is the total number of distinct maximal simplices of $K_n$.}
\label{tab:cover-max-counts}
\end{table}

For algorithmic purposes one may also record the auxiliary based counts
\[
m_{\star}^{\mathrm{based}}(n),\qquad
m_{\top}^{\mathrm{based}}(n),\qquad
m_e^{\mathrm{based}}(n),
\]
obtained before global deduplication of coincident simplices. These are not intrinsic invariants of $K_n$; they belong instead to the local presentation layer of the computation. Table~\ref{tab:based-max-counts} records their values for $1\le n\le 25$.

\begin{table}[ht]
\centering
\small
\renewcommand{\arraystretch}{1.08}
\setlength{\tabcolsep}{6pt}
\begin{tabular}{rrrr}
\toprule
$n$ & $m_{\star}^{\mathrm{based}}$ & $m_{\top}^{\mathrm{based}}$ & $m_{e}^{\mathrm{based}}$ \\
\midrule
1 & 0 & 0 & 0 \\
2 & 0 & 0 & 2 \\
3 & 0 & 0 & 4 \\
4 & 3 & 0 & 4 \\
5 & 6 & 3 & 4 \\
6 & 15 & 6 & 4 \\
7 & 22 & 15 & 4 \\
8 & 41 & 30 & 4 \\
9 & 59 & 49 & 4 \\
10 & 91 & 83 & 4 \\
11 & 131 & 131 & 4 \\
12 & 191 & 196 & 4 \\
13 & 260 & 281 & 4 \\
14 & 369 & 416 & 4 \\
15 & 500 & 569 & 4 \\
16 & 676 & 795 & 4 \\
17 & 905 & 1070 & 4 \\
18 & 1208 & 1453 & 4 \\
19 & 1585 & 1919 & 4 \\
20 & 2083 & 2546 & 4 \\
21 & 2702 & 3298 & 4 \\
22 & 3498 & 4312 & 4 \\
23 & 4500 & 5531 & 4 \\
24 & 5759 & 7117 & 4 \\
25 & 7322 & 9020 & 4 \\
\bottomrule
\end{tabular}
\caption{Based maximal-simplex counts for $1\le n\le 25$, before global deduplication of coincident vertex sets.}
\label{tab:based-max-counts}
\end{table}

The discrepancy between Tables~\ref{tab:cover-max-counts} and~\ref{tab:based-max-counts} illustrates the bookkeeping issue emphasized earlier: local presentations are plentiful, but the passage to distinct global simplices involves substantial identification.

\subsection{Exact computational pipelines}

The numerical invariants above may be computed by several exact finite methods.

The first is direct clique enumeration in the graph $G_n$. This produces the counts $c_r(n)$ and hence the Euler characteristics $\chi_n$ by alternating summation.

The second is the maximal-container method from Section~4: enumerate the local admissibility data $A_{\max}(\lambda,c)$ and $C_{\max}(\lambda,a)$, construct the corresponding full star- and full top-simplices, enumerate the resulting local presentation counts, and deduplicate globally.

The third is nerve-side enumeration: compute the canonical cover by distinct full star- and full top-simplices, form the nerve $N_n$, and use the identity
\[
\chi_n=\chi(N_n).
\]

The present tables were generated by exact finite enumeration. For each $n\le 25$ we generated the vertex set of $G_n$ from the integer partitions of $n$, constructed the edge set exactly from admissible unit transfers, enumerated the cliques of $G_n$ directly, constructed the full star- and full top-simplices from the local admissibility data, deduplicated coincident vertex sets to obtain the distinct cover and maximal-simplex counts, and computed the nerve of the resulting canonical cover. Thus the clique tables come from direct clique enumeration in $G_n$, while the maximal-simplex and cover tables come from exact local star/top enumeration with global deduplication.

For each $n\le 25$, the reported value of $\chi_n$ was checked both by the alternating clique sum and by the nerve-side computation $\chi(K_n)=\chi(N_n)$. Likewise, $m_{\max}(n)$ was obtained by deduplicating the union of all maximal presentations, while the typed counts $m_{\star}(n)$, $m_{\top}(n)$, and $m_e(n)$ were recorded separately after type-wise deduplication.

For the extended clique table through $n=60$, the low-dimensional counts $c_1(n),\dots,c_5(n)$ were computed exactly by finite clique enumeration in the partition graph. In addition, the full clique profile at $n=60$, namely $c_6(60),\dots,c_{11}(60)$ together with the already displayed values $c_1(60),\dots,c_5(60)$, was computed in the same way in order to verify the global leader at that endpoint.

The machine-readable ancillary file accompanying this version contains the exact numerical rows reported in the tables, the extended low-dimensional clique data $c_1(n),\dots,c_5(n)$ through $n=60$, and the full clique profile at $n=60$. It also contains the values used to produce Figures~\ref{fig:clique-log-all} and~\ref{fig:c3-c4-crossover}.

In particular, every numerical claim in Section~5 is intended to be auditable from the manuscript together with this ancillary file, without any appeal to floating-point approximation or heuristic estimation. Agreement among the independent exact pipelines described above provides a strong internal consistency check on the reported data.

\subsection{Normalization conventions for sequence data}

If the resulting sequences are later compared with external databases or submitted to OEIS, several normalization conventions must remain fixed.

First, the indexing is
\[
n\ge 1.
\]
No value at $n=0$ is used in this paper.

Second, the sequences
\[
(\chi_n)_{n\ge 1}
\qquad\text{and}\qquad
(b_n)_{n\ge 1}
\]
must be kept separate. They differ only by the constant shift $1$, but they have different natural meanings.

Third, clique-size indexing must not be confused with simplex-dimension indexing. Thus $c_4(n)$ and $f_3(n)$ denote the same numerical sequence, but under two different grading conventions.

Fourth, distinct maximal-simplex counts must not be conflated with based counts. The former are intrinsic global quantities, whereas the latter belong to the pre-deduplication stage of the computation.

Finally, counts of maximal simplices should be kept separate from counts of cover elements in the canonical cover and from counts of arbitrary cliques. These layers are related, but they are not interchangeable.

\section{Conclusion and outlook}

In this paper we studied the clique complex
\[
K_n=\Cl(G_n)
\]
of the partition graph from the numerical side, taking as input the previously established homotopy classification
\[
K_n \simeq \bigvee^{\,b_n} S^2.
\]

The paper yields a consistent counting framework. It fixes a dictionary relating clique counts, simplex counts, Euler characteristic, reduced Euler characteristic, and the sequence $b_n$. It also isolates two counting languages for the same invariant,
\[
\chi(K_n)=\sum_{r\ge 1}(-1)^{r-1}c_r(n)
\]
and
\[
\chi(K_n)=\chi(N_n),
\]
where $N_n$ is the nerve of the canonical cover by full star- and full top-simplices. In addition, we organized the maximal-simplex layer into three distinct counting regimes: cover-side counts, intrinsic counts of genuine maximal simplices, and auxiliary based counts prior to deduplication.

On the computational side, the extended clique data show that the first $c_3$-to-$c_4$ crossover occurs at $n=60$ and that, at this endpoint, tetrahedra are the global leader among clique sizes.

This leads to a natural family of integer sequences attached to $K_n$. Its primary topological layer is formed by
\[
\chi_n=\chi(K_n)
\qquad\text{and}\qquad
b_n=\chi_n-1.
\]
A second layer is formed by the fixed-$r$ clique-count sequences $(c_r(n))_{n\ge 1}$. A third layer consists of the distinct maximal-simplex counts
\[
m_{\star}(n),\qquad m_{\top}(n),\qquad m_e(n),
\]
together, when desired, with the total count $m_{\max}(n)$ of distinct maximal simplices. A fourth cover-side layer is given by
\[
\nu_{\star}(n),\qquad \nu_{\top}(n),\qquad \nu_C(n)=|C_n|.
\]
The auxiliary based counts belong to the computational presentation layer and should be kept separate from these intrinsic sequences. Tables~\ref{tab:chi-bn-25}--\ref{tab:based-max-counts} record exact data for these layers over the range $1\le n\le 25$, while Table~\ref{tab:clique-counts-1-60-summary} and Figures~\ref{fig:clique-log-all}--\ref{fig:c3-c4-crossover} extend the low-dimensional clique-count layer through $n=60$.

We do not claim closed formulas here for $\chi_n$ or for the general clique-count family $c_r(n)$. Instead, the paper provides a framework in which those questions can be posed cleanly and attacked by finite computation, local-to-global enumeration, or nerve-side counting.

\subsection{OEIS-oriented outlook}

The numerical sequences arising from the complexes $K_n$ are natural candidates for future OEIS comparison and, where appropriate, submission. The most immediate examples are:

\begin{enumerate}
\item the Euler-characteristic sequence $(\chi_n)_{n\ge 1}$;
\item the wedge-multiplicity sequence $(b_n)_{n\ge 1}$, with $b_n=\chi_n-1$;
\item the fixed-$r$ clique-count sequences $(c_r(n))_{n\ge 1}$ for fixed $r$;
\item the maximal-simplex sequences $(m_{\star}(n))_{n\ge 1}$, $(m_{\top}(n))_{n\ge 1}$, and $(m_e(n))_{n\ge 1}$.
\end{enumerate}

For OEIS purposes, several normalization issues are essential. The indexing must be fixed explicitly, the distinction between $\chi_n$ and $b_n$ must be preserved, clique-size indexing must not be confused with simplex-dimension indexing, and auxiliary based counts must not be conflated with intrinsic distinct counts. Only after these conventions are fixed does sequence comparison become reliable.

\subsection{Open problems}

\begin{problem}
Find exact formulas, recurrences, or effective structural reductions for the Euler characteristic sequence
\[
\chi_n=\chi(K_n)
\]
and equivalently for
\[
b_n=\chi_n-1.
\]
\end{problem}

\begin{problem}
For fixed $r\ge 2$, study the clique-count sequence
\[
c_r(n),
\]
including exact computation, possible recurrences, and asymptotic growth. Even the first nontrivial cases $r=2,3,4$ already appear natural.
\end{problem}

\begin{problem}
Understand the passage from based counts to distinct counts in a more structural way. In particular, determine when different local presentations yield the same full star- or full top-simplex, and whether the resulting deduplication admits a conceptual classification.
\end{problem}

\begin{problem}
Compare systematically the three exact counting pipelines: direct clique enumeration in $G_n$, maximal-container enumeration via local admissibility data, and nerve-side enumeration through the canonical cover. Determine which invariants are most efficiently accessible in each framework.
\end{problem}

\begin{problem}
Study the global behaviour of the resulting sequences: monotonicity and non-monotonicity phenomena, first and second differences, growth rates, and possible regularities or oscillations across $n$. At present such questions are primarily empirical unless supported by proof.
\end{problem}

\section*{Acknowledgements}

The author acknowledges the use of ChatGPT (OpenAI) for discussion, structural planning, and editorial assistance during the preparation of this manuscript. All mathematical statements, proofs, computations, and final wording were checked and approved by the author, who takes full responsibility for the contents of the paper.

\end{document}